\documentclass[3p,times]{elsarticle}

\usepackage{amsmath}
\usepackage{amsthm}

\usepackage{amssymb}
\usepackage[figuresright]{rotating}

\usepackage[font=footnotesize,labelfont=bf]{caption}
\usepackage[font=footnotesize,labelfont=bf]{subcaption}
\newtheorem{defn}{Definition}
\newtheorem{prop}{Proposition}
\newtheorem{cor}{Corollary}
\newtheorem{exmp}{Example}
\newtheorem{thm}{Theorem}
\newtheorem{obs}{Observation}
\newtheorem{property}{Property}
\usepackage[english]{babel}
\usepackage{empheq,xcolor}

\begin{document}
	
	\begin{frontmatter}

		\title{Operational vs. Umbral Methods and Borel Transform}

		\author[Enea]{G. Dattoli}
		\ead{giuseppe.dattoli@enea.it}
		
		\author[Enea]{S. Licciardi \corref{cor}}
		\ead{silvia.licciardi@enea.it}

		\address[Enea]{ENEA - Frascati Research Center, Via Enrico Fermi 45, 00044, Frascati, Rome, Italy}
		\cortext[cor]{Corresponding author}

		\begin{abstract}
			Differintegral methods, currently exploited in calculus, provide a fairly unexhausted source of tools to be applied to a wide class of problems involving the theory of special functions and not only. The use of integral transforms of  Borel type and the associated formalism will be shown to be an effective means, allowing a link between umbral and operational methods.  We merge these two points of view to get a new and efficient method to obtain integrals of special functions and the summation of the associated generating functions as well.
		\end{abstract}

		
	\end{frontmatter}
	
	\section{Introduction}

Operational methods, developed within the context of the fractional derivative formalism \cite{Oldham}, have opened new possibilities in the application of Calculus.  Even classical problems, with well-known solutions, may acquire a different flavor, if viewed within such a perspective which, if properly pursued, may allow further progresses disclosing new avenues for their study and generalizations.\\

\noindent It is indeed well known that the operation of integration is the inverse of the derivation such a statement, by itself, does not enable a formalism to establish rules to handle integrals and derivatives on the same footing. \\

 \noindent An almost natural environment to place this specific issue is the formalism of real order derivatives in which the distinction between integrals and derivatives becomes superfluous. The use of the formalism associated with the fractional order operators offers new computational tools as e.g. the extension of the concept of integration by parts. Within such a context we can prove the following statement \cite{Comp} .

\begin{prop}\label{Prop1}
	The integral of a function $f\in C^\infty$ can be written in terms of the series 
	
	\begin{equation}\label{intfunct}
	F(x):=\int f(x)dx=\sum_{s=0}^{\infty}(-1)^s\dfrac{x^{s+1}}{(s+1)!}f^{(s)}(x), \quad \forall x\in Dom \left\lbrace  f(x)\right\rbrace , 
	\end{equation}
	where $f^{(s)}(x)$ denotes the $s^{th}$-derivative of the integrand function. 
\end{prop}
\noindent The relevant proof needs a few premises.

\begin{defn}\label{DerNegOp}
	$\forall x\in\mathbb{R}, \forall f\in C^\infty$, let
	\begin{equation*}
	\begin{split}\label{key}
	& \int g(x)f(x)dx= \hat{D}_x^{-1}\left(g(x)f(x )\right), \\
	& \hat{D}_x^{-1}s(x)=\int s(x)dx,\qquad \qquad g(x)=1.
	\end{split}
	\end{equation*}
	where $\hat{D}_x^{-1}$ is the \emph{negative derivative operator} and, for the operation of definite integration, we set
	
	\begin{equation}\label{NDO}
	{}_{\alpha}\hat{D}_x^{-1}s(x)=\int_{\alpha}^{x}s(\xi)d\xi =S(x)-S(a).	
	\end{equation}
\end{defn}

\begin{proof}[\textbf{Proof.}][Proposition \ref{Prop1}]
	We rewrite eq. \eqref{intfunct} according to Definition \ref{DerNegOp} and by the use of the \textit{Leibniz formula}, written as
	
	\begin{equation}\label{negder}
\hat{D}_x^{-1}\left(g(x)f(x )\right)=\sum_{s=0}^{\infty}\binom{-1}{s}g^{(-1-s)}(x)f^{(s)}(x).
	\end{equation}
	We note that, being
	
	\begin{equation}\label{key}
	\binom{-1}{s}=(-1)^s,
	\end{equation}
	 for $g(x)=1$, we get 
	
	\begin{equation}\label{key}
	 g^{(-1-s)}(x)=\hat{D}_x^{(-1-s)}1=\dfrac{x^{s+1}}{(s+1)!}
\end{equation}
thus eventually ending up with

\begin{equation*}\label{key}
\hat{D}_x^{-1}\left(1 \cdot f(x )\right)=\sum_{s=0}^{\infty}(-1)^s\dfrac{x^{s+1}}{(s+1)!}f^{(s)}(x)=\int f(x)dx\;.
\end{equation*}
\end{proof}

The interesting element of such an analytical tool is that it allows the evaluation of the primitive of a function in terms of an automatic procedure, analogous to that used in the calculus of the derivative of a function. At the same time, it marks the conceptual, even though not formal, difference between the two operations. It is implicit in eq. \eqref{NDO} that the the underlying computational procedure involves, most of the times, an infinite number of steps. Eq. \eqref{negder} becomes useful if, e.g., the function $f(x)$ has peculiar properties under the operation of derivation, like being cyclical, vanishing after a number of steps or other.\\

The formalism we have just envisaged can be combined, e.g.,  with the properties of the special polynomials to find useful identities as shown in the following examples.

\begin{exmp}
 In the case of two variable Hermite Kamp\'{e} d\'{e} F\'{e}ri\'{e}t polynomials \cite{Appell}

\begin{equation} \label{GrindEQ__5_} 
H_{n} (x,y)=n!\, \sum _{r=0}^{\lfloor\frac{n}{2} \rfloor}\frac{x^{n-2\, r} y^{r} }{(n-2\, r)!\, r!} ,  \quad \forall x,y\in\mathbb{R}, \forall n\in\mathbb{N}
\end{equation} 
satisfying the properties \cite{Babusci}

\begin{equation}\begin{split} \label{GrindEQ__6_} 
& \partial_x^{s}\; H_{n} (x,y)=\frac{n!}{(n-s)!} H_{n-s} (x,\, y), \quad \forall s\in\mathbb{N}: 0\leq s\leq n, \\ 
& \partial_y^{s}\; H_{n} (x,y)=\frac{n!}{(n-2s)!} H_{n-2s} (x,\, y), \quad \forall s\in\mathbb{N}: 0\leq s\leq \frac{n}{2},
 \end{split} \end{equation} 
we obtain the definite integrals

\begin{equation}\begin{split}\label{key}
& \int H_{n}  (x ,\, y)\, dx\; =
% \; \int _{0}^{x}H_{n}  (\xi ,\, y)\, d\xi=
\sum _{s=0}^{n}\frac{(-1)^{s} x^{s+1} }{(s+1)!}  \frac{n!}{\left(n-s\right)\, !} H_{n-s} (x,y)=
\sum _{s=0}^{n} \binom {n} {s}\; \frac{(-1)^{s} x^{s+1}}{(s+1)}    H_{n-s} (x,y), \\ 
&  \int H_{n}  (x ,\, y)\, dy\; =
%  \;\int _{0}^{y}H_{n}  (x,\, \eta )\, d\eta =
\sum _{s=0}^{\lfloor\frac{n}{2}\rfloor}\frac{(-1)^{s} y^{s+1}  }{(s+1)!}  \frac{n!}{\left(n-2s\right)\, !} H_{n-2s} (x,y).
\end{split}
\end{equation}
Furthermore, by noting that the repeated integrals of Hermite polinomials with respect to the $x$ variable is 

\begin{equation}\label{IntHs}
H_n^{(-s)\mid_x}(x,y)=\frac{n!}{(n+s)!}H_{n+s}(x,y),
\end{equation}
we find\footnote{It should be noted that the eq. \eqref{IntHs} is consistent with the $s^{th}$ primitive in the sense that 

\begin{equation*}\label{key}
\left( \frac{d}{dx}\right) ^s H_n^{(-s)\mid_x}(x,y)=H_n(x,y).
\end{equation*}
It is evident that such a definition excludes an undetermined polynomial in $x$ of order $s$ originated by the iterated constants emerging in the process of successive integration.}

\begin{equation}\label{key}
\int H_{n}  (x ,\, y)\,H_{m}  (x ,\, y)\, dx\; =\sum_{s=0}^n\frac{(-1)^s \;n!\;m!}{(n-s)!(m+1+s)!}H_{n-s}(x,y)H_{m+1+s}(x,y).
\end{equation}
and 

\begin{equation}\label{key}
\int H_{n}  (x ,y)\,H_{m}  (x , y)x^p\,dx =\sum_{s=0}^{\min(m,n)} \frac{(-1)^s\;p!}{(p+s+1)!}x^{p+s+1}\sum_{r=0}^s\binom{s}{r}\frac{n!\;m!}{(n-r)!(m-(s-r))!}H_{n-r}(x,y)H_{m-(s-r)}(x,y).
\end{equation}
\end{exmp}

\begin{exmp}
Integrals involving the product of Hermite and circular functions 

\begin{equation}\begin{split}\label{key}
& \int H_{n}  \, (x,\, y)\, \cos (x )\, d\, x =-\sum _{s=0}^{n} \cos \left(x+(s+1)\dfrac{\pi }{2} \right) \frac{n!}{(n-s)!}  H_{n-s} (x,y)\, ,\\ 
& \int H_{n}  \, (x,\, y )\, \cos (y )\, d\, y =-\sum _{s=0}^{\lfloor\frac{n}{2}\rfloor}\cos \left(y+(s+1)\dfrac{\pi }{2} \right) \frac{n!}{(n-2s)!}  H_{n-2s} (x,y). 
\end{split}\end{equation}
can be solved or, in the case of definite integrals, we can find

\begin{equation}\begin{split}\label{key}
& \int_a^x H_{n}  \, (\xi\,, y)\, \cos (\xi)\, d\, \xi =F(x)-F(a)\\ 
& \int_a^y H_{n}  \, (x,\, \eta )\, \cos (\eta )\, d\, \eta =F(y)-F(a),\\
& F(x):=\int f(x)\;dx
 \end{split}\end{equation}
which, for $a=0$, specializes as

\begin{equation}\begin{split}\label{key}
\int_a^x H_{n}  \, (\xi\,, y)\, \cos (\xi)\, d\, \xi & 
=-\sum _{s=0}^{n} \cos \left(x+(s+1)\dfrac{\pi }{2} \right) \frac{n!}{(n-s)!}  H_{n-s} (x,y)- n!\left| \cos\left( (n+1)\frac{\pi}{2}\right) \right|(-1)^{\lfloor\frac{n-1}{2}\rfloor}e_{\lfloor\frac{n-1}{2}\rfloor}(-y),\\
  e_m(x)&=\sum_{r=0}^{m}  \frac{x^r}{r! } \; truncated\;exponential\;function,\\
\int_a^y H_{n}  \, (x,\, \eta )\, \cos (\eta )\, d\, \eta 
& =-\sum _{s=0}^{\lfloor\frac{n}{2}\rfloor} \cos \left(y+(s+1)\dfrac{\pi }{2} \right) \frac{n!}{(n-2s)!}  H_{n-2s} (x,y)-n! (-1)^{-\frac{n-2}{4}}\;{}_{[4]} e_{n-2}\left((-1)^{\frac{1}{4}}x \right), \\
  {}_{[k]}e_m(x)&=\sum_{r=0}^{\lfloor\frac{m}{k}\rfloor}  \frac{x^{m-kr}}{(m-kr)! } \; truncated\;exponential\;function\;of \;order\; k
 \end{split}\end{equation}
\end{exmp}

\begin{exmp}
$\forall a,b,x\in\mathbb{R}	$ we find
	\begin{equation}\label{intexpH}
	\int e^{\;a\, x ^{2} +b\, x }  \, dx =\sum _{s=0}^{\infty }\frac{x^{s+1} }{(s+1)!}  H_{s} (-2\, a\, x-b,\, a)\, e^{\;a\, x^{2} +b\, x}.
	\end{equation} 
The proof is obtained  by taking into account that

\begin{equation} \label{GrindEQ__8_} 
\partial_x ^{s}\; e^{\;a\, x^{2} +b\, x} =H_{s} (2\, a\, x+b,\, a)\, e^{\;a\, x^{2} +b\, x}  
\end{equation} 
and knowing that $(-1)^nH_n(x,y)=H_n(-x,y)$ \cite{Babusci}.\\

\noindent Furthermore, by using Proposition \eqref{Prop1}, we obtain integrals involving products of a Gaussian and a cosine

\begin{equation}\label{key}
 \int e^{\;a\, x ^{2} +b\, x }  \, \cos (x )\, dx =\sum _{s=0}^{\infty }
%\frac{\cos \left(x+s\dfrac{\pi}{2} \right)}{(s+1)!}  H_{s} (-2\, a\, x-b,\, a)\, e^{a\, x^{2} +b\, x} .\\
(-1)^s \frac{x^{s+1}}{(s+1)!}\;F_s(x;a,b)\;e^{\;ax^2+bx},
\end{equation}
where 
\begin{equation}\label{key}
 F_n(x;a,b):=\sum_{r=0}^n \binom{n}{r}\;H_r(2ax+b,a)\cos\left(x+(n-r)\frac{\pi}{2} \right)
\end{equation} 
or, in a more compact form,

\begin{equation}\label{key}
F_n(x;a,b):=Re\left(e^{inx} H_n(2ax+b+i,a)\right). 
\end{equation}
\end{exmp}

 These are just few examples to introduce the flexibility of the method, in the forthcoming section we will combine this formalism with others of umbral and operational nature to get further tools to strength the relevant computational capabilities.

\section{Umbral Methods and the Negative Derivative Formalism}\label{UMNDF}

 In a numbers of previous papers \cite{On Ramanujan, Winternitz,DBabusci,SLicciardi} it has been established that the umbral image of a Bessel function is a Gaussian. This statement can be profitably exploited within the present context, provided we premise some concepts \cite{SLicciardi, S.Roman}.
 
 \begin{defn}
 	The function
 	
 	\begin{equation}\label{key}
 	\varphi(\nu):=\varphi_\nu=\frac{1}{\Gamma (\nu +1)}, \quad \forall \nu\in\mathbb{R}.
 	\end{equation}
 	is called umbral ``vacuum".
 \end{defn}
 
 \begin{defn}
 	We introduce the operator $\hat{c}$ called ``umbral",
 	
 \begin{equation}\label{key}
 \hat{c}:=e^{\partial_z},
 \end{equation}	
as  the vacuum shift operator, with $z$ the domain's variable of the function on which the operator acts.
 \end{defn}

\begin{thm}
	The umbral operator $\hat{c}^\nu$, $\forall \nu\in\mathbb{R}$, is the action of the operator $\hat{c}$ on the vacuum $\varphi_{0}$ such that\footnote{See ref. \cite{SLicciardi} for a rigorous treatment of the umbral concepts.}
	
\begin{equation}\label{key}
	\hat{c}^{\nu } \varphi _{0} =\frac{1}{\Gamma (\nu +1)}. 
\end{equation}
\end{thm}
\noindent Through these premises, we define the correspondence quoted below.

\begin{prop}
	$\forall x,\nu\in\mathbb{R}$, the umbral image of a $0$-order Bessel function is a Gaussian function
	
\begin{equation}\label{GrindEQ__10_} 
J_0(x)=e^{-\hat{c}\, \left(\frac{x}{2} \right)^{2} } \varphi _{0} .
\end{equation} 	
\end{prop}
 
 \begin{proof}
 $\forall x,\nu\in\mathbb{R}$, let
 
 \begin{equation}\label{key}
 u(x)=e^{-\hat{c}\, \left(\frac{x}{2} \right)^{2} } \varphi _{0}, 
 \end{equation}	
 the relevant series expansion proofs  that $u(x)=J_0(x)$, indeed

\begin{equation} \label{GrindEQ__11_} 
 u(x)=e^{-\hat{c}\, \left(\frac{x}{2} \right)^{2} } \varphi _{0} =\sum _{r=0}^{\infty }\frac{(-\hat{c})^{r} }{r!}  \left(\frac{x}{2} \right)^{2\, r} \varphi _{0}  =\sum _{r=0}^{\infty }\frac{(-1)^{r} }{\left(r!\right)^{2} }  \left(\frac{x}{2} \right)^{2\, r} =J_{0} (x).
\end{equation} 
%We have therefore defined the umbral form in eq. \eqref{GrindEQ__10_} as the $0$-order cylindrical Bessel function.
\end{proof}

\begin{cor}
The $n^{th}$ order counterparts are, within the present context, defined as

\begin{equation} \label{GrindEQ__12_} 
J_{n} (x)=\left(\hat{c}\, \frac{x}{2} \right)^{n} e^{-\hat{c}\, \left(\frac{x}{2} \right)^{2} } \varphi _{0} =\sum _{r=0}^{\infty }\frac{\left(-1\right)^{r} }{(n+r)!\, r!}  \left(\frac{x}{2} \right)^{n+2\, r}  .
\end{equation} 
\end{cor}

Such a restyling has allowed noticeable simplifications, concerning the handling of integrals of Bessel functions and of many other problems associated with the use of the Ramanujan Master Theorem (RMT) \cite{Ramanujian}. Furthermore, the same formalism provides the possibility of viewing at the relevant theory by using fairly elementary algebraic tools \cite{SLicciardi}. Without entering in the details of the proof of the RMT, we note that it can be formulated as a kind of general rule, which will be stated as it follows.

\begin{thm}\label{thmcorr}
	 If an umbral correspondence is established between two different functions, such a correspondence can be extended to other operations including derivatives and integrals.
	\end{thm}

\begin{exmp}
 According to the Theorem \ref{thmcorr}, since an umbral correspondence between Gaussian and Bessel functions has been established, we can state further identities deriving from such identification. The well known Gaussian integral $\int _{-\infty }^{\infty }e^{-a\, x^{2} }  dx=\sqrt{\frac{\pi }{a} }$ can be exploited to get an analogous result for the Bessel function

\begin{equation} \begin{split}\label{GrindEQ__13_} 
&  \int _{0}^{\infty }J_{0}  (x)\, dx=\int _{0}^{\infty }e^{-\hat{c}\, \left(\frac{x}{2} \right)^{2} } dx \, \varphi _{0} =
\dfrac{1}{2}\sqrt{\frac{4\pi }{\hat{c}} } \varphi _{0}  =\sqrt{\pi } \hat{c}^{-\frac{1}{2} } \varphi _{0} =\sqrt{\pi } \frac{1}{\Gamma \left(1-\frac{1}{2} \right)} =1, \\ 
& \Gamma (x)=\int _{0}^{\infty }e^{-t}  t^{x-1} dt,\, \, \quad Re(x)>0 .
 \end{split} \end{equation} 
The use of the same rule, a kind of \emph{Principle of permanence of the formal properties},
%(a close consequence of the RMT)
 yields the further identity \cite{D.Babusci,nopubl} 

\begin{equation}\label{intJ0}
 \int _{0}^{\infty }J_{0}  (x)\, x^{\nu -1} dx=\int _{0}^{\infty }e^{-\hat{c}\, \left(\frac{x}{2} \right)^{2} }  x^{\nu -1} dx\, \varphi _{0} =2^{\nu -1} \frac{\Gamma \left(\frac{\nu }{2} \right)}{\Gamma \left(1-\frac{\nu }{2} \right)} ,\qquad 0< \nu<\frac{3}{2}
\end{equation}
 and 

\begin{equation}\label{key}
\int _{0}^{\infty }J_{0}  (x^{2} )\, dx=4^{-\frac{3}{4} } \frac{\Gamma \left(\frac{1}{4} \right)}{\Gamma \left(\frac{3}{4} \right)}.
\end{equation}
\end{exmp}

\begin{exmp}
Other infinite integrals can be obtained by a judicious use of the same ``principle'', therefore we find

\begin{equation} \begin{split}\label{GrindEQ__15_} 
 \int _{0}^{\infty }e^{-x^{2} }  J_{0} (bx)\, dx&=\int _{0}^{\infty }e^{-x^{2} }  e^{-\hat{c}\, \left(\frac{b\, x}{2} \right)^{2} } \, dx\, \varphi _{0} =\frac{\sqrt{\pi } }{2} \frac{1}{\sqrt{1+\frac{b^{2} }{4} \hat{c}} } \varphi _{0}  =
\frac{ \sqrt{\pi }}{2}  \sum _{r=0}^{\infty }\binom{-\frac{1}{2} }{r} \, \left(\frac{b}{2} \right)^{2r} \hat{c}^{r} \varphi _{0} =\\
 & =\frac{\pi }{2} \sum _{r=0}^{\infty }\frac{1}{\, \Gamma \left(\frac{1}{2} -r\right)\, (r!)^{2} }  \, \, \left(\frac{b}{2} \right)^{2r} .
 \end{split}
  \end{equation} 
 \end{exmp}

\begin{prop}
	$\forall x\in\mathbb{R}$, $\forall n\in\mathbb{N}$, we set
	
	\begin{equation}\label{key}
	\int _{0}^{x}J_{0}  \, (\xi )\, d\, \xi =\sum _{s=0}^{\infty }\frac{x^{s+1} }{(s+1)}  \left(\sum _{r=0}^{\lfloor\frac{s}{2}\rfloor }\frac{(-1)^{r} }{(2x)^r\;r!\, (s-2\, r)!}  J_{s-r} (x)\right).
	\end{equation}
\end{prop}
\begin{proof}
We use of the properties of the Gaussian functions under repeated derivatives, as e.g. the fact that

\begin{equation} \label{GrindEQ__16_} 
\left(\frac{d}{dx} \right)^{n} e^{a\, x^{2} } =H_{n} (2\, a\, x,a)\, e^{a\, x^{2} }  ,
\end{equation} 
which allows the derivation of its Bessel-umbral counterpart identity \cite{Babusci}

\begin{equation}\label{key}
\left(\frac{d}{dx} \right)^{n} J_{0} (x)=\left(\frac{d}{dx} \right)^{n} e^{-\hat{c}\, \left(\frac{x}{2} \right)^{2} } \varphi _{0} =H_{n} \left( -\, \hat{c}\, \frac{x}{2} ,-\frac{\hat{c}}{4} \right)  e^{-\hat{c}\, \left(\frac{x}{2} \right)^{2} } \varphi _{0}=
(-1)^{n} n!\sum _{r=0}^{\lfloor\frac{n}{2} \rfloor}\frac{(-1)^{r} }{(2x)^r\;r!\, (n-2\, r)!}  J_{n-r} (x). 
\end{equation}
So, we can ``translate'' the identity \eqref{intfunct} as

\begin{equation*}\label{key}
\int _{0}^{x}J_{0}  \, (\xi )\, d\, \xi =\sum _{s=0}^{\infty }(-1)^s\frac{x^{s+1}}{(s+1)!}\partial_{x}^{(s)}J_0(x)=\sum _{s=0}^{\infty }\frac{x^{s+1} }{(s+1)}  \left(\sum _{r=0}^{\lfloor\frac{s}{2}\rfloor }\frac{(-1)^{r} }{(2x)^r\;r!\, (s-2\, r)!}  J_{s-r} (x)\right).
\end{equation*}
\end{proof}

\begin{obs}
The extension of such a point of view to the theory of Hankel transform \cite{Prudnikov} may be particularly illuminating. Limiting ourselves to the $0$-order case we set

\begin{equation} \label{GrindEQ__19_} 
H_{0} (f(r))(s)=\int _{0}^{\infty }r\, f(r)\, J_{0} (sr)\, dr =\int _{0}^{\infty }r\, f(r)\, e^{-\, \hat{c}\, \left(\frac{s\, r}{2} \right)^{2} } \, dr \, \varphi _{0}  
\end{equation} 
therefore, in the case of the Hankel transform of the function $f(r)=\frac{e^{-r^{2} } }{r} $, we obtain

\begin{equation} \label{GrindEQ__20_} 
H_{0} \left(\frac{e^{-r^{2} } }{r} \right)(s)=\left(\int _{0}^{\infty }e^{-\left(1+\frac{s^{2} }{4} \hat{c}\right)\, r^{2} } \, dr \right)\, \varphi _{0} =\frac{\sqrt{\pi } }{2} \frac{1}{\sqrt{1+\frac{s^{2} }{4} \hat{c}} } \varphi _{0},  
\end{equation} 
thus reducing the problem to the integral evaluated in eq. \eqref{GrindEQ__15_}.
\end{obs}

\begin{exmp}
 Further integral transforms can be framed within the same context, e.g. the well-known identities
 
\begin{equation} \begin{split}\label{GrindEQ__21_} 
& \int _{0}^{\infty }J_{0}  (2\sqrt{x\, u} )\, \sin (u)du=\cos (x), \\ 
& \int _{0}^{\infty }J_{0}  (2\sqrt{x\, u} )\, \cos (u)du=\sin (x).
\end{split} \end{equation} 
They can be stated quite easily by noting that 

\begin{equation}\label{eq30}
J_{0} (2 \sqrt{x} )=C_{0} (x)=\sum _{r=0}^{\infty }\frac{(-x)^{r} }{\left(r!\right)^{2} }  
\end{equation}  
and by establishing the umbral correspondence

\begin{equation}\label{key}
C_0(x)=e^{-\hat{c}x}\varphi_{0}
\end{equation}
which allows the following handling of the integrals in eq. \eqref{GrindEQ__21_} 
\begin{equation} \label{GrindEQ__22_} 
\int _{0}^{\infty }J_{0}  (2\sqrt{x\, u} )\, e^{i\, u} du=\int _{0}^{\infty }e^{-(\hat{c}\, x-i)\, u}  du\, \varphi _{0} =\frac{1}{\hat{c}\, x-i} \varphi _{0} =i\, \sum _{r=0}^{\infty }(-i\, \hat{c}\, x)^{r}  \varphi _{0} =i\, e^{-i\, x} .
\end{equation} 
\end{exmp}

\begin{exmp}
Furthermore, $\forall x\in\mathbb{R}$, $\forall s\in\mathbb{N}$ we can set

\begin{equation}\label{key}
x^{\frac{s}{2}}\int _{0}^{\infty } C_{s}  (x\, u)\, e^{iu}du=x^{\frac{s}{2}} i\;\sum_{r=0}^\infty \frac{(-ix)^r}{(r+s)!}=(-1)^s \;i^{1-s} x^{-\frac{s}{2}}\left( e^{-ix}-e_{s-1}(-ix)\right) 
\end{equation}
where $C_{s} (x)$ is the $s$-order Tricomi function, satisfying the identity \cite{Germano}

\begin{equation}\begin{split} \label{GrindEQ__23_} 
& \left(\frac{d}{dx} \right)^{s} C_{0} (x)=(-1)^{s} C_{s} (x), \\ 
& C_{s} (x)=\sum _{r=0}^{\infty }\frac{(-x)^{r} }{r!\, (r+s)!},  
\end{split} 
\end{equation} 
or, by expanding the range of $s$, 

\begin{equation}\label{key}
x^{\frac{s}{2}}\int _{0}^{\infty } C_{s}  (x\, u)\, e^{iu}du=(-1)^s \;i^{1-s} x^{-\frac{s}{2}}\;\frac{e^{-ix}\left(\Gamma(s)-\gamma(s,-ix) \right) }{\Gamma(s)}.
\end{equation}
The results are obtained by using eq. \eqref{eq30}, Proposition \eqref{Prop1} and eq. \eqref{GrindEQ__22_} .\\

\noindent We can so deduce the integrals

  \begin{equation}\begin{split}
  & x^{\frac{s}{2}}\int _{0}^{\infty } C_{s}  (x\, u)\, \sin (u)du=\int _{0}^{\infty }\frac{J_s\left(2\sqrt{xu} \right) }{u^{\frac{s}{2}}}\sin(u)du=x^{\frac{s}{2}} \; \frac{{}_1F_2\left[ 1;\frac{1}{2}+\frac{s}{2},1+\frac{s}{2};-\frac{x^2}{4}\right] }{s!}\\ 
  & x^{\frac{s}{2}}\int _{0}^{\infty } C_{s}  (x\, u)\, \cos (u)du=\int _{0}^{\infty }\frac{J_s\left(2\sqrt{xu} \right) }{u^{\frac{s}{2}}}\cos(u)du=x^{\frac{s}{2}+1} \; \frac{{}_1F_2\left[ 1;1+\frac{s}{2},\frac{3}{2}+\frac{s}{2};-\frac{x^2}{4}\right] }{(s+1)!}
  \end{split} \end{equation} 
and, by a straightforward application of eqs. \eqref{GrindEQ__11_} and \eqref{eq30} we also obtain

\begin{equation}\label{key}
\int _{0}^{\infty }C_{0}  (x\, u)\, J_{0} (u)du=J_{0} (x).
\end{equation}
\end{exmp}
 What we have described so far displays the ``wild'' potentialities of the method, further elements proving its reliability yielding significant possibilities, which we are going to discuss in the forthcoming sections.

\section{Borel Transform}

 The theory of integral transforms is one of the pillars of operational calculus \cite{Prudnikov}. On the other side many transforms have been shown to be expressible in terms of exponential operators, as shown in the case of the Fractional Fourier and Airy transforms \cite{Borel}.\\

 In this section we make a further step in this direction by developing new analytical tools to reformulate the theory of the Borel transform ($BT$) \cite{Babusci} 
 
\begin{equation} \label{GrindEQ__25_} 
f_{B} (x)=\int _{0}^{\infty }e^{-t}  f(t\, x)\, d t 
\end{equation} 
which has played a significant role in treating the series resummation in quantum field theory \cite{DelFranco}.

\begin{prop}
 The use of the identity \cite{Babusci} 
 
\begin{equation} \label{GrindEQ__26_} 
f(t\, x)=t^{x\, \partial _{x} } f(x) 
\end{equation} 
allows to write \cite{Babusci}

\begin{equation} \label{GrindEQ__27_} 
\begin{split}
& f_{B} (x)=\hat{B}\, \left[f(x)\right]\, ,\\ 
& \hat{B}=\int _{0}^{\infty}e^{-t} t^{x\, \partial _{x} }  dt=\Gamma (x\, \partial _{x} +1).
 \end{split} 
\end{equation} 
\end{prop}

\begin{exmp}
According to such a procedure it is shown e.g. that the Borel transform of the $0$-order Tricomi function is just provided by \cite{SLicciardi}.

\begin{equation} \label{GrindEQ__28_} 
\hat{B}\, \left[C_{0} (x)\right]=\Gamma (x\, \partial _{x} +1)\, \left[C_{0} (x)\right]=\sum _{r=0}^{\infty }(-1)^{r}  \Gamma (r+1)\, \frac{x^{r} }{\left(r!\right)^{2} }  =e^{-x} . 
\end{equation} 
The $\hat{B}$ operator has evidently acted on the Bessel type function $C_0(x)$ by
providing a kind of ``downgrading" from higher transcendental function to
the ``simple" exponential. The successive application of the Borel operator to the same
previous function produces the further result reported below \cite{SLicciardi}.

\begin{equation} \label{GrindEQ__29_} 
\hat{B}^{2} \left[C_{0} (x)\right] =\hat{B}[e^{-x}]=\sum _{r=0}^{\infty }(-1)^{r}  \Gamma (r+1)\, \frac{x^{r} }{r!} =
\frac{1}{1+x}, \qquad  \left|x\right|<1.
\end{equation} 
Again, we notice the same behaviour: the exponential function has been reduced to a rational function.
The further application of $\hat{B}$ yields a diverging series, namely

\begin{equation} \label{GrindEQ__30_} 
\hat{B}^{3} \, \left[C_{0} (x)\right]=\sum _{r=0}^{\infty}(-1)^{r}  r!\, x^{r} , \qquad \forall x\in\mathbb{R}. 
\end{equation} 
\end{exmp}
We have interchanged Borel operators and series summation without taking too much caution. In the case of eq. \eqref{GrindEQ__28_} such a procedure is fully justified, in eq. \eqref{GrindEQ__29_} the method is limited to the convergence region while in the case of eq. \eqref{GrindEQ__30_} it is not justified since it gives rise to a diverging series. We will take in the following some freedom in handling these problems and include in our treatment also the case of diverging series.\\

 Since the repeated application of $BT$ is associated with the Borel operator raised to some integer power, we explore the possibility of defining a fractional $BT$ and more in general a real power positive and negative $BT$.

\begin{defn}
 We introduce the operator

\begin{equation} \label{GrindEQ__31_} 
\hat{B}_{\alpha } =\int _{0}^{\infty }e^{-t} t^{\alpha \, x\, \partial _{x} }  dt=\Gamma \left(\alpha \, x\, \partial _{x} +1\right), \qquad \forall \alpha\in\mathbb{R}
\end{equation} 
which will be referred as the $\alpha$-order Borel transform. 
\end{defn}

\begin{exmp}
We find that the $\frac{1}{2}$-order applied to the $0$-order Bessel function yields \cite{SLicciardi}

\begin{equation}\label{GrindEQ__32_}
\hat{B}_{\frac{1}{2} } \left[J_{0} (x)\right]=
\Gamma \left(\frac{1}{2} x\, \partial _{x} +1\right)\sum _{r=0}^{\infty }\frac{(-1)^{r} }{\left(r!\right)^{2} }  \left(\frac{x}{2} \right)^{2 r} =
\sum _{r=0}^{\infty }\frac{(-1)^r }{r!}  \left(\frac{x}{2} \right)^{2 r} =e^{-\left(\frac{x}{2} \right)^{2} } .
\end{equation}
\end{exmp}

\begin{exmp}
By assuming, for $\alpha>0$, that an operator $\left(\hat{B}_{\alpha } \right)^{-1} $such that 

\begin{equation} \label{GrindEQ__33_} 
\begin{split}
& \left(\hat{B}_{\alpha } \right)^{-1} \hat{B}_{\alpha } =\hat{1}, \\ 
& \left(\hat{B}_{\alpha } \right)^{-1} =\frac{1}{\Gamma \left(\alpha \, x\, \partial _{x} +1\right)} ,\qquad  \forall \alpha\in\mathbb{R}
\end{split}
\end{equation} 
exists, we can invert eq. \eqref{GrindEQ__32_} and write 

\begin{equation}\label{key}
\left(\hat{B}_{\frac{1}{2} } \right)^{-1} \left[e^{-\left(\frac{x}{2} \right)^{2} } \right]=J_{0} (x).
\end{equation}
\end{exmp}

A more rigorous definition of the inverse of the operator $\hat{B}_{\alpha}$ may be achieved through the use of the Hankel contour integral \cite{Prudnikov}, namely 

\begin{equation}\label{key}
\frac{1}{\Gamma (z)} =\frac{i}{2\, \pi } \int _{C} \frac{e^{-t} }{(-t)^{z } }  \, dt, \qquad \left|z\right|<1,
\end{equation}
which can be exploited to write

\begin{equation}\label{key}
\left(\hat{B}_{\alpha } \right)^{-1} f(x)=\frac{i}{2\, \pi } \int _{C}\frac{e^{-t} }{t}  f\left(\frac{x}{(-t)^{\alpha } } \right)\, dt.
\end{equation}
 After the previous remarks we can state the following Theorem \cite{SLicciardi}.

\begin{thm}
Let $f(x)$ a function such that $\int _{-\infty }^{+\infty }f(x)\, dx=k$, $\forall k\in\mathbb{R}$, then 

\begin{equation}\label{key}
\int _{-\infty }^{+\infty }\hat{B}_{\alpha } \left[f(x)\right] \, dx=k\, \Gamma (1-\alpha ), \qquad \mid \alpha \mid<1.
\end{equation}
\end{thm}
\begin{proof}
The proof is fairly straightforward by applying eq. \eqref{GrindEQ__25_} and the variable change $t^\alpha x=\sigma$.
 $\forall k\in\mathbb{R}$, $\mid \alpha \mid<1$, we find

\begin{equation*} \label{GrindEQ__38_} 
\begin{split}
 \int _{-\infty }^{+\infty }\hat{B}_{\alpha } \left[f(x)\right] \, dx&=
\int _{-\infty }^{+\infty }\left( \int _{0}^{\infty }e^{-t}  f(t^{\alpha } \, x) \, dt\right) dx= \int _{-\infty}^{+\infty }\; e^{-t} \left( \int _{0}^{\infty }f(t^{\alpha } x)  \, dx\right)  dt=\\
& =\int _{-\infty}^{+\infty }\; e^{-t} t^{-\alpha } \left( \int _{0}^{\infty }f(\sigma )  \, d\sigma \right)  dt=\int _{-\infty}^{+\infty } f(\sigma)\left( \int _{0}^{\infty }e^{-t}t^{-\alpha}dt\right)d\sigma 
=k\, \Gamma (1-\alpha ).
\end{split} 
\end{equation*} 
\end{proof}
The same procedure can be exploited for cases involving the inverse transform.\\

\noindent It is evident that the previous Theorem can be applied to the derivation of the integrals of Bessel functions.

\begin{exmp}
 For this specific case we get

\begin{equation} \label{GrindEQ__39_} 
\int _{-\infty }^{+\infty }\hat{B}_{\frac{1}{2} } \left[J_{0} (x)\right]dx =I_{J_{0} } \Gamma \left( \frac{1}{2} \right) , 
\end{equation} 
where $I_{J_{0} }$ is the integral of the Bessel function assumed to be unknown and once derived from  eq. \eqref{GrindEQ__32_} yields

\begin{equation} \label{GrindEQ__40_} 
I_{J_{0} } =\left(\Gamma \left( \frac{1}{2} \right) \right)^{-1} \int _{-\infty }^{+\infty }e^{-\left(\frac{x}{2} \right)^{2} } dx =2 .
\end{equation} 
\end{exmp}

The problems we have touched so far will be more deeply discussed in the forthcoming sections where we will further extent the concept of Borel transform and state the link with previous researches.

\section{Generalization of Borel Transforms and their Applications to Special Functions}

\noindent The use of Borel transform techniques is a widely exploited tool in Analysis and in applied science. In quantum field theory it became a tool of crucial importance to deal with problems concerning series resummation in renormalization theories \cite{DelFranco,Kleinert,Elizalde}. In the following we will adopt and generalize some of the techniques developed within such a research framework to provide a different formal environment for the umbral formalism discussed in Sec. \ref{UMNDF}. \\

\noindent To better clarify the relevance of Borel resummation methods to the topics of the present article, we critically review what we did in the previous section. The successive application of the Borel transform to a function with non-zero radius of convergence has led to a function whose series coefficients exhibit an exponential growth with the order of expansion. Its range of convergence is therefore getting smaller and smaller, eventually vanishing.  In quantum field theory one is faced with the opposite problem, namely that of recovering a function with non-zero radius of convergence starting from a diverging series.  The problem is ``cured'' by dividing each term in the expansion by a factor $k!$.  In this way we obtain a Borel sum, if it can be analytically summed and analytically continued over the whole real axis, it is called Borel summable \cite{E.Borel,GHHArdy}. After these remarks, we can state the following.

\begin{prop}
	  The $0$-order Bessel function is an inverse of $\frac{1}{2}$-order Borel sum of a Gaussian, $\forall x\in\mathbb{R}$, 

\begin{equation}\label{key}
\left(\hat{B}_{\frac{1}{2} } \right)^{-1} e^{-\left(\frac{x}{2} \right)^{2} } =\left(\hat{B}_{\frac{1}{2}} \right)^{-1} \sum _{r=0}^{\infty }\; f_{r} \left( \frac{x}{2}\right) ^{2 r}  =\sum _{r=0}^{\infty }f_r\frac{1 }{r!} \left(\frac{x}{2} \right)^{2\, r}  ,\qquad f_{r} =\frac{(-1)^{r} }{r!} .
\end{equation}
\end{prop}
 The previous statement is essentially a rewording of the Umbral definition of the Bessel functions discussed in Sec. \ref{UMNDF}. As a preliminary conclusion we will adopt the following simplification.

\begin{thm}
In umbral context, the $\alpha$-order Borel anti transform of a function $f(x)$, characterized by the formal series expansion $f(x)=\sum _{r=0}^{\infty }f_{r}  x^{r} $, can be written as 

\begin{equation} \label{GrindEQ__48_} 
\left(\hat{B}_{\alpha} \right)^{-1} \left[f(x)\right]=\sum _{r=0}^{\infty }f_{r}  \left(\hat{c}^\alpha\, x\right)^{r} \varphi _{0} , \qquad \forall \alpha\in \mathbb{R} 
\end{equation} 
and for all the operations of integration, derivative, series summation and so on, the operator $\hat{c}$ can be treated as an ordinary constant.
\end{thm}

\noindent The Laguerre polynomials can be framed within the same context too, and indeed the following identity is easily understood .

\begin{prop}
$\forall x,y\in\mathbb{R}$, $\forall n\in\mathbb{N}$, let

\begin{equation}\label{key}
\begin{split}
& \left(y-x\right)^{n} =\int _{0}^{\infty }e^{-t}  L_{n} (x\, t,\, y)dt, \\ 
& L_{n} (x,y)=\sum _{r=0}^{n}\binom{n}{r}\frac{(-1)^{r} }{r!}  x^{r} y^{n-r} .
\end{split}
\end{equation} 
\end{prop}
\begin{proof}
$\forall x,y\in\mathbb{R}$, $\forall n\in\mathbb{N}$

\begin{equation*}\label{key}
\left(y-x\right)^{n} =\sum_{r=0}^n \binom{n}{r}(-x)^r y^{n-r} \frac{\int_{0}^\infty e^{-t}t^r dr}{r!}=\int_{0}^\infty e^{-t}L_n(xt,y).
\end{equation*}	
\end{proof}

\begin{cor}
 A consequence of the previous relation is that 

\begin{equation} \label{GrindEQ__43_} 
\begin{split}
& \sum _{r=0}^{\infty }\xi ^{r} (y-x)^{r}  =\frac{1}{\left(1-y\, \xi \right)+\xi \, x} =\Gamma \left(x\, \partial _{x}+1 \right)\, G(x,y|\xi ), \\ 
& G(x,y|\xi )=\sum _{r=0}^{\infty }\xi ^{r} L_{r}  (x,y) \end{split} 
\end{equation} 
Accordingly, the generating function of the Laguerre polynomials is the inverse of the Borel transform of the geometric series, namely

\begin{equation}\label{key}
G(x,y|\xi)=\frac{1}{\Gamma (1+x\, \partial _{x} )} \left[\frac{1}{\left(1-y\, \xi \right)+\xi \, x} \right]=\frac{1}{\left(1-y\, \xi \right)+\xi \, \hat{c}\, x}\; \varphi _{0}=\frac{1}{1-y\, \xi } e^{-\frac{x\, \xi }{1-y\, \xi } } .
\end{equation}
\end{cor}

\begin{prop}
 If the Borel operator acts on the $y$ variable of Laguerre polynomial, we obtain $\forall x,y\in\mathbb{R}$, $\forall n\in\mathbb{N}$,

\begin{equation}\label{key}
\hat{B}\left[ L_n(x,y)\right]=b_{n} (x,\, y) ,
\end{equation}
 where $b_{n} (x,y)$ denotes the Bessel truncated polynomials \cite{Trunc}.
 \end{prop}
\begin{proof}
\begin{equation*}\label{key}
\hat{B}\left[ L_n(x,y)\right]=\int _{0}^{\infty }e^{-t}  L_{n} (x\, ,\, y\, t)\;dt=\Gamma\left(y\partial_{y}+1 \right)\left[ L_n(x,y)\right]= n!\sum _{r=0}^{n}\frac{(-1)^{r} }{\left(r!\right)^{2} }  x^{r} y^{n-r}=b_{n} (x,\, y) .
\end{equation*}
\end{proof}

We can make further progress in the understanding of their properties by exploiting the Borel transform formalism. The use of  the following identity \cite{Trunc}

\begin{equation} \label{GrindEQ__46_} 
\sum _{n=0}^{\infty }\frac{\xi ^{n} }{n!}  L_{n} (x,\, y)=e^{y\xi } C_{0} (x\xi ), \qquad \forall \xi\in\mathbb{R} 
\end{equation} 
 yields e.g. the corresponding generating function for the Bessel truncated polynomials 

\begin{equation}\label{key}
\sum _{n=0}^{\infty }\frac{\xi ^{n} }{n!}  b_{n} (x,\, y)=\int _{0}^{\infty }e^{-t}  e^{y\, t\xi } dt\, C_{0} (x\xi )=\frac{C_{0} (x\;\xi )}{1-y\, \xi}, \qquad \forall \xi\in\mathbb{R} :\mid \xi \mid<\frac{1}{\mid y \mid}.
\end{equation}

\begin{exmp}
	A fairly immediate consequence of the previous result is their use for the derivation of the lacunary generating functions of Laguerre polynomials.
 According to our formalism we find indeed

\begin{equation}\label{key}
\begin{split}
\sum _{n=0}^{\infty }\xi^{n}  L_{2\, n} (x,\, y)&=\left(\hat{B}\right)^{-1} \sum _{n=0}^{\infty }\xi ^{n}  \left[y-x\right]^{2n} =\left(\hat{B}\right)^{-1} \left[\frac{1}{1-\xi \, \left(y-x\right)^{2} } \right] =\frac{1}{2} \left(\hat{B}\right)^{-1} \left[\frac{1}{1-\sqrt{\xi } \left(y-x\right)} +\frac{1}{1+\sqrt{\xi } \left(y-x\right)} \right]= \\ 
& =\frac{1}{2} \left[\frac{1}{\left(1-\sqrt{\xi } y\right)\, \left(1+\frac{\hat{c}\, \sqrt{\xi } x}{1-\sqrt{\xi } y} \right)} +\frac{1}{\left(1+\sqrt{\xi } y\right)\, \left(1-\frac{\hat{c}\, \sqrt{\xi } x}{1+\sqrt{\xi } y} \right)} \right]\, \varphi _{0} = \\ 
& =\frac{1}{2} \left[\frac{1}{\left(1-\sqrt{\xi } y\right)\, } e^{-\frac{\sqrt{\xi } x}{1-\sqrt{\xi } y} } +\frac{1}{\left(1+\sqrt{\xi } y\right)\, } e^{\frac{\sqrt{\xi } x}{1+\sqrt{\xi } y} } \right]. \qquad \forall y\in\mathbb{R}:\;\mid y\mid\;<\frac{1}{\sqrt{\xi}}.
 \end{split}   
\end{equation}
\end{exmp}
 In the concluding remarks we will further comment on the Laguerre lacunary generating functions and on the relevant link with previous researches \cite{Bab2}.\\

 To make further progresses in our discussion, we introduce the following integral transform, also known in the literature as the Borel-Leroy ($BL$) transform \cite{Zinn}
 
\begin{equation}\label{key}
{}_\gamma\hat{B}_{\alpha} \left[f(x)\right]=\int _{0}^{\infty }e^{-t}  t^{\gamma -1} f\left(t^{\alpha } x\right) dt
\end{equation}
 and the associated differential operator can also be written as

\begin{equation}\label{key}
{}_\gamma\hat{B}_{\alpha}=\Gamma (\gamma +\alpha \, x\, \partial _{x} ).
\end{equation}

\begin{exmp}
The relevant action on the $\gamma$-order Tricomi function yields e.g.

\begin{equation}\label{key}
\begin{split}
& {}_{\gamma+1}\hat{B}_{\alpha} \left[C_{\gamma } (x)\right]=e_{\left( \alpha,\gamma \right) } (-x), \\ 
& e_{\left( \alpha,\gamma \right) }(-x)=\sum _{r=0}^{\infty }\frac{(-x)^{r}}{r!}\frac{ \Gamma (\alpha \, r+\gamma+1)}{\Gamma (r+\gamma+1)}  
 \end{split}
\end{equation}
while the $BL$ anti-transform of the exponential yields the Bessel-Wright function (which is a generalization of the Tricomi-Bessel function), namely

\begin{equation}\label{key}
\begin{split}
& \left({}_{\gamma+1}\hat{B}_\alpha \right)^{-1} \left[e^{-x} \right]=W_{\left( \alpha,\;\gamma \right) }(-x)=\sum _{r=0}^{\infty }\frac{(-x)^{r} }{r!}\frac{1}{\Gamma (\alpha \, r+\gamma +1)}  ,\\
& W_{\left( \alpha,\;\beta \right) } (x)=\sum_{r=0}^\infty\frac{x^r}{r!\Gamma(\alpha r +\beta+1)}, \qquad \forall x\in\mathbb{R}, \forall \alpha,\beta\in\mathbb{R}_0^+ .
\end{split}
\end{equation}
\end{exmp}

\begin{exmp}
 As a further example of generalization we discuss the B-Borel transform 
\begin{equation} \label{GrindEQ__40b_} 
{}_{\gamma}\hat{BB}_\alpha^{(\beta,\delta)} \left[f(x)\right]=
\int _{0}^{1}  \left(1-t\right)^{\beta -1}t^{\gamma -1} f\left(t^{\alpha } \left(1-t\right)^{\delta } x\right)\, \, dt 
\end{equation} 
which, upon the use of the B-Euler function \cite{L.C.Andrews}, can be transformed in the differential form

\begin{equation}\label{key}
\begin{split}
& {}_{\gamma}\hat{BB}_\alpha^{(\beta,\delta)} =B(\gamma +\alpha \, x\, \partial _{x} ,\, \beta +\delta \, x\, \partial _{x} ), \\ 
& B(\alpha ,\beta )=\frac{\Gamma (\alpha )\, \Gamma (\beta )}{\Gamma (\alpha +\beta )} .
 \end{split}
\end{equation}
It is interesting to note that the previous operator acts on a Bessel-Wright $W_{(\alpha,0)}(x)$ by increasing the order namely, e.g., 

\begin{equation}\label{key}
{}_1\hat{BB}_\alpha^{(\beta,0)}\left[ \frac{W_{(\alpha,0)}(x)}{\Gamma(\beta)}\right] =\sum _{r=0}^{\infty }\frac{x^r }{r!\Gamma (\alpha r+\beta +1)} =W_{(\alpha,\beta)}(x),
\end{equation}
while, on an ordinary exponential function by transforming it into a Mittag-Leffler \cite{Mittag-Leffler}, namely

\begin{equation} \label{GrindEQ__42b_} 
\begin{split}
& {}_1\hat{BB}_1^{(\beta,0)}\left[  \frac{e^{x} }{\Gamma (\beta )} \right] = E_{\left( 1 ,\;\beta +1\right) } (x)=\sum _{r=0}^{\infty }\frac{x^r }{\Gamma ( r+\beta +1)}   ,\\
& E_{\left( \alpha,\;\beta\right) }(x)=\sum_{r=0}^\infty\frac{x^r}{\Gamma(\alpha r+\beta)}, \qquad \forall x\in\mathbb{R}, \forall \alpha,\beta\in\mathbb{R}^+ .
\end{split}
\end{equation} 
It is now evident that, regarding the evaluation of infinite integrals, if $\int_{0}^1 f(x)dx=k$ and through the identity in eq. \eqref{GrindEQ__40b_}, the following conclusion is ensured  

\begin{equation}\label{BBg}
\int _{-\infty }^{+\infty }{}_{\gamma}\hat{BB}_\alpha^{(\beta,\delta)} \left[f(x)\right] \, dx=k\, B\left(\gamma -\alpha ,\, \beta -\delta \right).
\end{equation}
A slight generalization of the method yields the further identity

\begin{equation}\label{key}
\int _{-\infty }^{+\infty }E_{\left( 1,\; \beta +1\right) } (-x^{2} )\, dx= \frac{\pi  }{\Gamma \left(\beta+\frac{1}{2}\right)} .
\end{equation}
\end{exmp}

These are just few examples of the wide possibilities offered by the present formalism further applications will be discussed in the forthcoming sections.

\section{The Case of Hermite Polynomials}

 In a paper of few years ago, Gessel and Jayawant \cite{Gessel} have discussed a triple lacunary generating function for Hermite polynomials. The authors employ two different methods, one of umbral nature, the other based on combinatorial arguments. In this section we reconcile the method of ref. \cite{Gessel} with the technique discussed in this paper. In particular we will take some advantage from the possibility of formulating the theory of Hermite polynomials by using the same point of view we adopted for the Laguerre family.\\

\noindent The $2$-variable polynomials, defined by the series \cite{Germano}

\begin{equation}\label{hermite}
H_{n} (x,\, y)=n!\, \sum _{r=0}^{\lfloor\frac{n}{2} \rfloor}\frac{x^{n-2\, r} y^{r} }{(n-2\, r)!\, r!}  
\end{equation}
belong to an Hermite-like family. This set of polynomials has many generalization and sometimes there is some confusion in the literature, regarding the relevant notation. For reasons which will be clarified in the following, they should be denoted by $H_{n}^{(2)} (x, y)$ and should be referred to as ``second order two-variable Hermite polynomials'', we will however keep the upper index only for polynomials with order $\ge 3$ or add it whenever strictly necessary to avoid confusion.\\

\noindent The generating  function reads

\begin{equation}\label{genfH}
\sum _{n=0}^{\infty }\frac{t^{n} }{n!}  H_{n} (x,\, y)=e^{x\, t+y\, t^{2} } .
\end{equation}
A remarkable property is the operational definition \cite{Germano}

\begin{equation} \label{GrindEQ__47b_} 
H_{n} (x,\, y)=e^{y\, \partial _{x}^{2} } x^{n}  
\end{equation} 
due to the fact that they are a solution of the partial differential equation
\begin{equation} \label{GrindEQ__48b_} 
\partial _{y} F(x,y)=\partial _{x}^{2} F(x,y), 
\end{equation} 
with the ``initial condition''

\begin{equation}\label{key}
F(x,\, 0)=x^{n} .
\end{equation}
For this reason they are also defined ``heat polynomials'' \cite{Widder}. From eq. \eqref{hermite} we derive the boundary condition at $x=0$
\begin{equation} \label{GrindEQ__50b_} 
H_{n} (0,\, y)=n!\frac{y^{\frac{n}{2} } }{\Gamma \left(\frac{n}{2} +1\right)} \, \left|\cos \left(n\, \frac{\pi }{2} \right)\right| .
\end{equation} 
The use of eq. \eqref{GrindEQ__47b_} is extremely useful, for example we can derive straightforwardly the double lacunary Hermite generating function, which can be formally written as 

\begin{equation}\label{GrindEQ__51a_}
\sum _{n=0}^{\infty }\frac{t^{n} }{n!}  H_{2\, n} (x,\, y)=e^{y\, \partial _{x}^{2} } e^{x^{2} t} .
\end{equation}
A definite meaning to the rhs of eq. \eqref{GrindEQ__51a_} is obtained through the application of the Gauss-Weierstrass transform \cite{Prudnikov}

\begin{equation} \label{GrindEQ__52_} 
e^{y\, \partial _{x}^{2} } f(x)=\frac{1}{2\, \sqrt{\pi \, y} } \int _{-\infty }^{+\infty }e^{-\frac{(\xi-x )^{2} }{4\, y} }  f(\xi )\, d\xi  
\end{equation} 
which allows the derivation of the generating function \eqref{GrindEQ__51a_} according to the following expression

\begin{equation} \label{GrindEQ__53_} 
\sum _{n=0}^{\infty }\frac{t^{n} }{n!}  H_{2\, n} (x,\, y)=\frac{1}{\sqrt{1-4\, y\, t} } e^{\frac{x^{2} t}{1-4\, y\, \, t} } ,\qquad \left|t\right|<\frac{1}{4\, \left|y\right|} 
\end{equation} 
which is sometimes called Doetsch rule \cite{Gessel2}.

\begin{exmp}
 The procedure we have adopted to derive eq. \eqref{GrindEQ__53_} can be extended to get the following generalization of the Doetsch rule 
 
\begin{equation}\label{key}
\begin{split}
\sum _{n=0}^{\infty }\frac{t^{n} }{n!}  H_{2\, n+l} (x,\, y)&=e^{y\partial _{x}^{2} } \left( x^{l} e^{ x^{2} t } \right)=\frac{1}{\sqrt{1-4\, y\, t} } e^{\frac{x^{2} t}{1-4\, y\, \, t} } \frac{H_{l} \left(\frac{x}{\sqrt{1-4\, y\, t} } ,\, y\right)}{\left(1-4\, y\, t\right)^{\frac{l}{2} } } = \\ 
& =\frac{1}{\sqrt{1-4\, y\, t} } e^{\frac{x^{2} t}{1-4\, y\, \, t} } H_{l} \left(\frac{x}{1-4\, y\, t} ,\, \frac{y}{1-4\, y\, t} \right), \qquad \left|t\right|<\frac{1}{4\, \left|y\right|} .
 \end{split}
\end{equation}
\end{exmp}

What we have described so far is what can be derived within the ``classical'' operational framework, it can be extended to higher order lacunary generating functions as we will show in the forthcoming sections. Symbolic methods can be equally efficient to formulate the theory of Hermite polynomials in a Laguerre fashion. 

\begin{defn}
We introduce the umbral operator

\begin{equation} \label{GrindEQ__55_} 
{}_y\hat{h}^{r} \theta _{0} :=\theta_r=\frac{y^{\frac{r}{2} } \, r!}{\Gamma \left(\frac{r}{2} +1\right)} \left|\cos \left(r\frac{\pi }{2} \right)\right| =\left\lbrace \begin{array}{ll} 
0 & r=2s+1 \\ y^s \dfrac{(2s)!}{s!} & r=2s
\end{array} \right. \forall s\in\mathbb{Z}.
\end{equation} 
\end{defn}
The operator ${}_y\hat{h}^r \theta_0$ has been defined in refs. \cite{SLicciardi,Motzkin,Gegenbauer} where it has been shown that Hermite polynomials can be written as

\begin{equation}\label{GrindEQ__56_}
H_{n} (x,\, y)=\left(x+{}_y\hat{h} \right)^{\, n} \theta _{0} , \qquad \forall x,y\in\mathbb{R}, \forall n\in\mathbb{N}. 
\end{equation}
According to eq. \eqref{GrindEQ__56_} the Hermite polynomials are reduced to the $n^{th}$ power of a binomial. All the relevant properties can be obtained by handling eq. \eqref{GrindEQ__56_} by means of elementary algebraic tools. The exponentiation of the umbra ${}_y\hat{h}$ will be particularly important in the present context. We note therefore the followinf properties. 

\begin{property}
	\begin{enumerate}
\item 
\begin{equation}\label{eq80}
e^{{}_y\hat{h}\, z} \theta_{0} =\sum _{r=0}^{\infty }\frac{\left({}_y\hat{h}\, z\right)^{\, r} }{r!} \theta_{0}  =e^{y\, z^{2} } .
\end{equation} 

\item 
\begin{equation}\label{key}
e^{{}_y\hat{h}^{2}z} \theta_{0} =\frac{1}{\sqrt{\pi } } \int _{-\infty }^{+\infty }e^{-\xi ^{2} +2\,  {}_y\hat{h} \sqrt{z}\;\xi } d\xi\;\theta_0 =\frac{1}{\sqrt{1-4\, y\, z} } , \qquad  \left|z\right|<\frac{1}{4\, \left|y\right|} 
\end{equation}
which is just a consequence of eq. \eqref{eq80} if we note that

\begin{equation}\label{key}
\frac{1}{\sqrt{\, \pi } } \int _{-\infty }^{+\infty }e^{-\xi ^{2} +2\, \xi \, {}_y\hat{h} \sqrt{z} } \,  d\xi \, \theta_0=\frac{1}{\sqrt{\, \pi } } \int _{-\infty }^{+\infty }e^{-\xi ^{2} (1-4\, y\, z)} \,  d\xi .
\end{equation}
\end{enumerate}
\end{property}

\begin{exmp}
The Doetsch formula and its extension can be therefore derived by the use of the identity \eqref{GrindEQ__56_}

\begin{equation}\label{key}
\sum _{n=0}^{\infty }\frac{t^{n} }{n!}  H_{2n} (x,\, y)=e^{(x+{}_y\hat{h})^{2} t} \theta_0=\frac{1}{\sqrt{\pi } } \int _{-\infty }^{+\infty }e^{-\xi ^{2} +2 \; \left( x+{}_y\hat{h}\right)\; \sqrt{t} \;
	 \xi } d\xi\; \theta_0.
\end{equation}

Let us now consider the same problem from a different point of view and write

\begin{equation} \label{GrindEQ__60_} 
 \sum _{n=0}^{\infty }\frac{t^{n} }{n!}  H_{2n} (x,\, y)=e^{(x+{}_y\hat{h})^{2} t} \theta_0=e^{x^{2} t} e^{{}_y\hat{h}^{2}\, t+2\, {}_y\hat{h}\, \, x\, t\, } \theta_0=e^{x^{2} t} \sum _{r=0}^{\infty }\frac{{}_y\hat{h}^{r} }{r!}  H_{r} (2\, x\, t,\, t)\, \theta_0 =e^{x^{2} t} \sum _{s=0}^{\infty }\frac{y^s }{s!}  H_{2s} (2\, x\, t,\, t)
\end{equation} 
which has been derived by using eqs. \eqref{genfH} and \eqref{GrindEQ__55_}. By comparing with the Doetch rule eq. \eqref{GrindEQ__53_}, we get the operational identity

\begin{equation}\label{key}
e^{{}_y\hat{h}^{2} \, t+2\, {}_y\hat{h}\, \, x\, t\, } \theta_0=\frac{1}{\sqrt{1-4\, t\, y} } e^{\frac{4 \, \left(t\, x\right)^{2}\;y }{1-4\, t\, y} } .
\end{equation}
\end{exmp}

Before going further let us note that the third order Hermite polynomials \cite{Babusci}

\begin{equation} \label{GrindEQ__17_} 
H_{n}^{(3)} (x,\, y,z)=n!\, \sum _{r=0}^{\lfloor\frac{n}{3} \rfloor}\frac{z^{r} H_{n-3\, r} (x,y)}{r!\, (n-3\, r)!}   
\end{equation} 
can be defined through the generating function

\begin{equation}\label{eq87}
\sum _{n=0}^{\infty }\frac{t^{n} }{n!}  H_{n}^{(3)} (x,\, y,\, z)=e^{x\, t+y\, t^{2} +z\, t^{3} } .
\end{equation}

\begin{thm}
	The triple lacunary generating function of second order Hermite polynomials, can be expressed in terms of double lacunary generating function of third order Hermite polynomials as
	
	\begin{equation}\label{key}
	\sum _{n=0}^{\infty }\frac{t^{n} }{n!}  H_{3\, n} (x,\, y)=e^{x^{3}t } \sum _{s=0}^{\infty }\frac{y^s}{s!}  H_{2s}^{(3)} (3\, x^{2} t,\, 3\, x\, t ,t).
	\end{equation}
\end{thm}	
\begin{proof}
On account of eqs. \eqref{GrindEQ__56_} we can write the triple lacunary Hermite generating function as

\begin{equation}\label{key}
\sum _{n=0}^{\infty }\frac{t^{n} }{n!}  H_{3\, n} (x,\, y)=e^{\left( x+{}_y\hat{h}\right) ^{3} t } \;\theta_0
\end{equation}
which, according to eq. \eqref{eq87}, allows the following conclusion 

\begin{equation} \label{GrindEQ__20b_} 
\sum _{n=0}^{\infty }\frac{t^{n} }{n!}  H_{3\, n} (x,\, y)=e^{x^{3}t } \sum _{r=0}^{\infty }\frac{{}_y\hat{h}^{r} }{r!}  H_{r}^{(3 )} (3\, x^{2} t,\, 3\, x\, t ,t)\, \theta_0=e^{x^{3}t } \sum _{s=0}^{\infty }\frac{y^s}{s!}  H_{2s}^{(3)} (3\, x^{2} t,\, 3\, x\, t ,t).
\end{equation} 
\end{proof}

\section{Associated Hermite Polynomials }

\begin{defn}
	 In analogy  with the case of Laguerre polynomials \cite{L.C.Andrews}, we introduce the \textit{associated Hermite polynomials} which, according to the present formalism read, $\forall x,y\in\mathbb{R}$ and $\forall n,p\in\mathbb{N}$,

\begin{equation}\label{key}
H_{n} (x,\, y|p)={}_y\hat{h}^{p} \left(x+{}_y\hat{h}\right)^{\, n} \theta_0=
\sum _{r=0}^{n}\binom{n}{r}\frac{x^{n-r}y^{\frac{r+p}{2}}(r+p)!  }{\Gamma \left(\frac{r+p}{2} +1\right)}  \left|\cos \left((r+p)\frac{\pi}{2} \right)\right|.
\end{equation}
\end{defn}

They cannot be identified with the generalized heat polynomials \cite{LevySt} which, within the present context, deserve a separate treatment. According to the previous definition we can ``state'' the following.

\begin{property}
	\begin{itemize}
		\item Index-duplication formula

\begin{equation} \label{GrindEQ__22b_} 
H_{2\, n} (x,\, y)=(x+{}_y\hat{h})^{n} (x+{}_y\hat{h})^{n} \theta_0=n!\, \sum _{s=0}^{n}\frac{x^{n-s} }{(n-s)!\, s!}  H_{n} (x,\, y|s).
\end{equation} 
	\item  Argument-duplication formula

\begin{equation} \label{GrindEQ__23b_} 
H_{n} (2\, x,\, y)=\left[\left(x+\frac{{}_y\hat{h}}{2} \right)+\left(x+\frac{{}_y\hat{h}}{2} \right)\right]^{\, n} \theta_0=\, \sum _{s=0}^{n}\binom{n}{s} \sum _{r=0}^{s}\binom{s}{r} x^{r}  H_{n-s} \left(x,\, \left. \frac{y}{2^2} \right| s-r\right).
\end{equation} 

\item

\begin{equation} \label{GrindEQ__24_} 
x^{n} =\left[\left(x+{}_y\hat{h}\right)-{}_y\hat{h}\right]^{\, n} \, \theta_0=\, \sum _{r=0}^{n} \binom{n}{r}(-1)^r H_{n-r} (x,y|r) .
\end{equation} 

\item

\begin{equation} \label{GrindEQ__25b_} 
H_{n+m} ( x, y)=(x+{}_y\hat{h})^{m}  (x+{}_y\hat{h})^{n}\; \theta_0= \sum _{r=0}^{m}\binom{m}{r} x^{m-r} H_{n} (x,\, y|r).
\end{equation} 
The last identity is a reformulation of the Nielsen Theorem, concerning the sum of the indices of Hermite polynomials \cite{Nielsen}.
\end{itemize}
\end{property}

The formalism we have just envisaged can be exploited to get a common framework to frame different families of polynomials sporadically in the literature, as discussed below.

\begin{exmp}
	$\forall x,y,\alpha,\beta\in\mathbb{R}$, $\forall n\in\mathbb{N}$, let
\begin{equation} \label{GrindEQ__26b_} 
H_{n} (x,\, y\, |\;\beta ;\alpha )= {}_y\hat{h}^{\;\beta } \left(x+{}_y\hat{h}^{\;\alpha } \right)^{\, n} \theta_0 
\end{equation} 
yields a family of polynomials with generating function

\begin{equation} \label{GrindEQ__27b_} 
\begin{split}
& \sum _{n=0}^{\infty }\frac{t^{n} }{n!}  H_{n} (x,\, y\, |\;\beta ;\alpha )=e^{x\, t} y^{\frac{\beta }{2} } {}_{\left( \frac{1}{2},\frac{1}{2}\right) }e_{\left( \alpha,\;\beta\right) } (y^{\frac{\alpha }{2} } \, t), \\ 
& {}_{\left( \frac{1}{2},\frac{1}{2}\right) }e_{\left( \alpha,\;\beta\right) } (x)=\sum _{r=0}^{\infty }\frac{x^r}{r!}\frac{\Gamma (\alpha \, r+\beta +1) }{\Gamma \left(\frac{\alpha \, r+\beta }{2} +1\right)}  \, \left|\cos \left(\alpha \, r+\beta\right) \frac{ \pi}{2} \right|. \end{split} 
\end{equation} 
Their properties can easily be studied and they are framed within the context of the Sheffer family. They can accordingly be defined through the operational rule

\begin{equation}\label{key}
H_{n} (x,\, y\, |\; \beta ;\alpha )=\, y^{\frac{\beta }{2} } {}_{\left( \frac{1}{2},\frac{1}{2}\right) }e_{\left( \alpha,\;\beta\right) } (y^{\frac{\alpha }{2} } \, \partial _{x} )\, x^{n} .
\end{equation}
\end{exmp}

In the following part of the paper we will discuss further elements characterizing the usefulness of the umbral point of view to the theory of Hermite polynomials.

\section{Umbra and Higher Order Hermite Polynomials}

The higher order Hermite polynomials are defined through the operational identity \cite{Babusci}

\begin{equation}\label{key}
H_{n}^{(m)} (x,y)=e^{y\, \partial _{x}^{m} } x^{n} ,\qquad \forall m\in\mathbb{N},
\end{equation}
 specified by the series

\begin{equation} \label{GrindEQ__30b_} 
H_{n}^{(m)} (x,\, y)=n!\, \sum _{r=0}^{\lfloor\frac{n}{m} \rfloor}\frac{x^{n-m\, r} y^{r} }{(n-m\, r)!\, r!}  
\end{equation} 
with generating function 

\begin{equation}\label{key}
\sum_{n=0}^\infty \frac{t^n}{n!}H_n^{(m)}(x,y)=e^{xt+yt^m}, \qquad \forall t\in\mathbb{R}.
\end{equation}

\begin{obs}
Through our formalism, the higher order Hermite polynomials can be reduced to the $n^{th}$ power of a binomial by introducing the umbral operator ${}_y\hat{h}_m^r$ such that, $\forall r\in\mathbb{R}$,

\begin{equation}\label{key}
\begin{split}
& {}_{y} \hat{h}_m^{r}\; {}_m\theta_0:={}_m\theta_r=\frac{y^{\frac{r}{m} } \, r!}{\Gamma \left(\frac{r}{m} +1\right)} A_{m,r}\; , \\ 
& A_{m,r} =\left\lbrace \begin{array}{l} 1 \\  0 \end{array}\right. \;\; r=ms \;\; \begin{array}{l} s\in\mathbb{Z}\\  s\notin\mathbb{Z} \end{array}
\end{split}
\end{equation}
which, indeed, allows to define them as

\begin{equation}\label{key}
H_{n}^{(m)} (x,y)=(x+{}_{y} \hat{h}_m)^{n} {}_m\theta _{0}
\end{equation}
and that, for $m=2$, provides eq. \eqref{GrindEQ__56_}. It is clearly evident that not too much effort is necessary to study the relevant properties, which can be derived by using the same procedure adopted for the second order case.
\end{obs}

\begin{obs}
 We can combine the Hermite umbral operators to get further generalizations, as for the three variable third order Hermite polynomials which, according to the previous formalism, can be defined as

\begin{equation}\label{key}
H_{n}^{(3)} (x,\, y,z)=(x+ {}_{y} \hat{h}_{2} +{}_{z} \hat{h}_3)^{n} {}_2\theta_{0,y} \;{}_3\theta_{0,z} \;.
\end{equation}
 Thereby we find
 
\begin{equation}\label{key}
\begin{split}
& H_{n}^{(3)} (x,\, y,z)=\sum _{s=0}^{n}\binom{n}{s}x^{n-s}\;{}_{y,z} \hat{h}_{2,3}^{s}\;   {}_2\theta_{0,y} \;{}_3\theta_{0,z},\\ 
& {}_{y,z} \hat{h}_{2,3}^{s}=\left( {}_{y} \hat{h}_{2} +{}_{z} \hat{h}_3\right)^s =\sum _{r=0}^{s}\binom{s}{r}\;
 {}_z \hat{h}_{3}^{s-r}  {}_{y} \hat{h}_{2}^{r}\;.
 \end{split}
\end{equation}
Obviously, the same result can be obtained through the other way

\begin{equation}\label{key}
H_{n}^{(3)} (x,\, y,z)=\sum _{s=0}^{n}\binom{n}{s} H_{n-s}(x,y) \;{}_{z} \hat{h}_3^s \;{}_3\theta_{0,z}.
\end{equation}
\end{obs}

\begin{obs}
The extension of the method to bilateral generating functions is quite straightforward too. We consider indeed the generating function

\begin{equation} \label{GrindEQ__35_} 
G_H(x,\, y;\, z,\, w|\, t)=\sum _{n=0}^{\infty }\frac{t^{n} }{n!}  H_{n} (x,\, y)\, H_{n} (z,\, w)=
\sum _{n=0}^{\infty }\frac{t^{n} }{n!}  (x+{}_y\hat{h}\,)^{n} H_{n} (z,\, w)\;\theta_{0,y} =
e^{\;z\; t\;(x+{}_y\hat{h})+w\; \left(t\;(x+ {}_y\hat{h} ) \right)^{\, 2}} \theta_{0,y}.
\end{equation} 
The use of our technique yields 

\begin{equation}\label{key}
G_H(x,\, y;\, z,\, w|\, t)=e^{\;x\;t\;(z\;+\;w\;x\;t)}
\frac{1}{\sqrt{1-4\;w\;t^{2} \;y} }
 e^{\frac{\left( z\;+\;2\;w\;t\;x\right) ^2\, t^{2} y}{1-4\, w\; t^{2}\; y} }, \qquad \mid 4\, w\; t^{2}\; y \mid<1.
\end{equation}
\end{obs}

Exotic generating functions involving e.g. products of Laguerre and Hermite polynomials can also be obtained and will be discussed elsewhere.\\

 The use of umbral methods looks much promising to develop a new point of view on the theory of special polynomials and of special functions as well.  Just to provide a flavor of the directions along which may develop future speculations, we consider the definition of the following umbral operator.
 
 \begin{defn}\label{HTheta}
 $\forall x,y\in\mathbb{R}$, $\forall n\in\mathbb{N}$, we get the position
 
\begin{equation}\label{key}
\hat{H}_{(x,y)}^{n}\; \Theta_0=H_{n} (x,\, y).
\end{equation}
 \end{defn}

\begin{prop}
According to Definition \ref{HTheta} we can write

1) \begin{equation} \label{GrindEQ__38b_} 
e^{t\, \hat{H}_{(x,y)} } \;\Theta_0=e^{x\, t+y\, t^{2} }  
\end{equation} 
and 

2) \begin{equation}\label{GrindEQ__39b_}
\sum _{n=0}^{\infty }\frac{t^{n} }{n!} H_{2\, n}  (x,\, y)=\frac{1}{\sqrt{1-4\, y\, t} } e^{\frac{x^{2} t}{1-4\, y\, t} } .
\end{equation}
\end{prop}
\begin{proof}
1) \begin{equation*}\label{key}
e^{t\, \hat{H}_{(x,y)} } \;\Theta_0 =\sum_{r=0}^\infty \frac{t^r}{r!}\hat{H}_{(x,y)}^{r}\; \Theta_0=\sum_{r=0}^\infty \frac{t^r}{r!}H_r(x,y)=e^{x\, t+y\, t^{2} }  .
\end{equation*}	
2) \begin{equation*}  
\sum _{n=0}^{\infty }\frac{t^{n} }{n!} H_{2\, n}  (x,\, y)=
e^{t\, \hat{H}_{(x,y)}^{2} } \;\Theta_0=
\frac{1}{\sqrt{\, \pi } } \int _{-\infty }^{+\infty }e^{-\xi ^{2} +2\, \sqrt{t}\;\hat{H}_{(x,y)}\;\xi }  d\xi \, \Theta_0=
\frac{1}{\sqrt{\, \pi } } \int _{-\infty }^{+\infty }e^{-(1-4\, y\, \, t)\, \xi ^{2} +2 \, x\sqrt{t}\;\xi }  d\xi \, =\frac{1}{\sqrt{1-4\, y\, t} } e^{\frac{x^{2} t}{1-4\, y\, t} } .
\end{equation*} 
\end{proof}

\begin{exmp}
Let us now consider the following integral

\begin{equation}\label{intzH}
I\left(x,y\right)=\int _{-\infty }^{+\infty }e^{-z^{2} \hat{H}_{(x,-y)} \, }  dz\, \Theta_0
\end{equation}
 equivalent to

\begin{equation}\label{intzHb}
 I(x,y)=\int_{-\infty }^{\infty }e^{-x\, z ^{2}-y\, z^4 }  d\, z\, =\, \int_{0}^{\infty }e^{-x\, s\, -y\;s^{2}}  s^{-\, \frac{1}{2} } ds .
\end{equation}
The formal solution of the integral in eq. \eqref{intzH} is

\begin{equation} \label{GrindEQ__41_} 
I\left(x,y\right)=\sqrt{\pi } \, \hat{H}_{(x,-y)}^{-\, \frac{1}{2} }\; \Theta_0 .
\end{equation} 
We are therefore faced with the necessity of specifying the meaning of the Hermite umbral raised to a negative fractional power. The use of standard Laplace transform methods yields

\begin{equation} \label{GrindEQ__42_} 
\hat{H}_{(x,-y)}^{-\frac{1}{2} } \;\Theta_0=\frac{1}{\Gamma \left(\frac{1}{2} \right)} \int _{0}^{\infty }e^{-s\hat{H}_{(x,-y)} } s^{-\, \frac{1}{2} }  ds\;\Theta_0=\frac{1}{\sqrt{\pi } } \int _{0}^{\infty }e^{-x\, s\, -y\;s^{2} }  s^{- \frac{1}{2} } ds 
\end{equation} 
which correctly reproduces eq. \eqref{intzHb} and that, through the identity \cite{SLicciardi,HermCalculus}

\begin{equation}\label{key}
H_n(x,-y)=(2y)^{-\frac{n}{2}}e^{\frac{x^2}{8y}}D_{-n}\left(\frac{x}{\sqrt{2y}} \right) 
\end{equation}
involving the parabolic cylinder functions $D_n(x)$, provides the solution of the integrals \eqref{intzH}-\eqref{intzHb}
\begin{equation}\label{key}
I(x,y)=\pi\; (2y)^{-\frac{1}{4}}e^{\frac{x^2}{8y}}D_{-\frac{1}{2}}\left(\frac{x}{\sqrt{2y}} \right) .
\end{equation} 
\end{exmp}
This is quite a significant result which ensures that the formalism has a very high level of flexibility.

\begin{obs}
 It is also interesting to note that the ``Fourier'' transform 

\begin{equation}\label{GrindEQ__43b_}
\hat{F}_{(k,\;\beta)}\;[f(\hat{H})]\;\Phi_0=\int_{-\infty}^\infty \hat{f}(x)e^{\;i\;x\;\hat{H}_{(k,\, \beta) }}dx\;\Theta
\end{equation}
can be viewed as a kind of Gabor transform \cite{Gabor}. 
\end{obs}

 The technique we have discussed in the paper is tightly bound to the method of quasi-monomials developed in ref. \cite{Germano,Bell}, provided that we make the following identification
 
\begin{equation} \label{GrindEQ__44_} 
\left(x+{}_y\hat{h}\right)\to \left(x+2\, y\, \partial _{x} \right).
\end{equation} 
 The differential operator on the left is used to define the Hermite polynomials as
 
\begin{equation}\label{key}
H_{n} (x,y)=\left(x+2\, y\, \partial _{x} \right)^{n} 1.
\end{equation}

There are certain advantages offered by the umbral method with respect to the monomiality technique which are all associated with the fact that in the former case one  deals with commuting operators.

\section{Generalized Heat Polynomials and Final Comments}

 The generalized heat polynomials ($GHP$) \cite{Nasim,Rosenbloom} are known to satisfy the $PDE$
 
\begin{equation} \label{GrindEQ__45_} 
\left\lbrace \begin{split}
& \partial _{y} P_{n,\nu } (x,y)=\partial _{x}^{2} P_{n,\nu } (x,y)+\frac{2\, \nu }{x} \partial _{x} P_{n,\nu } (x,y), \\ 
& P_{n,\nu } (x,\, 0)=\frac{\Gamma \left(\nu +\frac{1}{2} \right)}{\Gamma \left(\nu +\frac{1}{2} +n\right)} x^{2\, n} 
 \end{split} \right. 
\end{equation} 
which is an extension of the ordinary diffusion equation. They are defined as

\begin{equation} \label{GrindEQ__46b_} 
P_{n,\, \nu } (x,y)=\Gamma \left(\nu +\frac{1}{2} \right)\, \sum _{r=0}^{n}\binom{n}{r}\, \frac{x^{2\, (n-r)} (4\, y)^{r} }{\Gamma \left(\nu +\frac{1}{2} +n-r\right)}  
\end{equation} 
and, for particular values of the index $\nu$, they can be espressed in terms of known polynomials, as noted below. \\

\begin{itemize}
\item  \textit{Case 1}\\
 For $\nu =\frac{1}{2}$ the polynomials \eqref{GrindEQ__46b_} reduce to 

\begin{equation} \label{GrindEQ__47_} 
P_{n,\, \frac{1}{2} } (x,y)=\, n!\, \sum _{r=0}^{n}\frac{x^{2\, (n-r)} (4\, y)^{r} }{r!\left[(n-r)!\right]^{2} }   =L_{n} \left(-x^2,4\, y \right) .
\end{equation} 

\item \textit{Case 2}\\
For $\nu =0$ we find the following correspondence in terms of Hermite polynomials

\begin{equation}\label{key}
P_{n,\, 0} (x,y)=4^n \frac{n!}{(2n)!}H_{2\, n} (x , y)   .
\end{equation}   
\end{itemize}
The relevant restyling in umbral terms could be particularly useful to derive straightforwardly further properties. 

\begin{defn}
We accordingly define the $GHP$ by using the following binomial image $\forall x,y,\nu\in\mathbb{R}$, $\forall n\in\mathbb{N}$ 

\begin{equation} \label{GrindEQ__49_} 
\begin{split}
&P_{n,\, \nu } (x,y)=\, (x^{2}\, \hat{d}_{\nu }+4\;y )^{n} \;\eta_0, \\ 
& \hat{d}_{\nu }^{r}\; \eta_0=\frac{\Gamma \left(\nu +\frac{1}{2} \right)}{\Gamma \left(\nu +\frac{1}{2} +r\right)} .
 \end{split} 
\end{equation} 
\end{defn}

A fairly straightforward consequence of such a procedure is the derivation of the relevant generating functions.

\begin{prop}
$\forall x,y,\nu,t\in\mathbb{R}$

\begin{equation}\label{GrindEQ__50_} 
\sum _{n=0}^{\infty }\frac{t^{n} }{n!}  P_{n,\, \nu } (x,y)=
\Gamma \left(\nu +\frac{1}{2} \right)e^{\;4\;y\;t } C_{\nu -\frac{1}{2} } \left(-\;t\;x^2\right)
\end{equation}
where $C_s(x)$ is the Tricomi-Bessel function in eq. \eqref{GrindEQ__23_}.
\end{prop}
\begin{proof}
\begin{equation*}\label{key}
\sum _{n=0}^{\infty }\frac{t^{n} }{n!}  P_{n,\, \nu } (x,y)=
\sum _{n=0}^{\infty }\frac{t^{n} }{n!}  (x^{2}\, \hat{d}_{\nu }+4\;y )^{n}\; \eta_0=
e^{\;4\;y\;t} e^{\;x^2\;t\;\hat{d}_{\nu } }\; \eta_0=
\Gamma \left(\nu +\frac{1}{2} \right)e^{\;4\;y\;t } C_{\nu -\frac{1}{2} } \left(-\;t\;x^2\right).
\end{equation*}
\end{proof}

In analogous way we can prove that

\begin{equation}\label{GrindEQ__51_} 
\sum _{n=0}^{\infty }t^{n}  P_{n,\, \nu } (x,y)=\frac{\Gamma \left(\nu +\frac{1}{2} \right)}{1-4yt}E_{\left( 1,\;\nu+\frac{1}{2}\right) }\left( \frac{x^2t}{1-4yt}\right) .
\end{equation}
where $E_{\alpha } (x)$ is the Mittag-Leffler function in eq. \eqref{GrindEQ__42b_}.

\begin{obs}
 By recalling that the Tricomi function is an eigenfunction of the operator \cite{Tricomi,Airy} 

\begin{equation}\label{key}
\hat{T}=\partial _{x} x\, \partial _{x} +\alpha \, \partial _{x} 
\end{equation}
in the sense that

\begin{equation}\label{key}
\hat{T}C_{\alpha } (\lambda \, x)=-\lambda \, C_{\alpha } (\lambda \, x),
\end{equation}
we find, from the eq. \eqref{GrindEQ__50_}, the identity 

\begin{equation}\label{ppn}
 \partial _{y} P_{n,\nu } (x,y)=\frac{2}{x}
\left[\dfrac{1}{2}\;{}_l\partial_x +\left( \nu-\frac{1}{2}\right) \partial_x
 \right]P_{n,\nu } (x,y)  
\end{equation}
where ${}_l\partial_x=\partial_x x\partial_x$ is the Laguerre derivative \cite{Babusci}. Eq. \eqref{ppn} is complementary to eq. \eqref{GrindEQ__45_}.
\end{obs}

We conclude this paper by going back to the use of umbral formalism, by discussing a particularly interesting example.\\

\noindent We introduce the operator

\begin{equation}\label{ppi}
\hat{p}_m^r\; \pi_0:=r!\delta_{r,m},
\end{equation}
with $\delta_{\alpha,\beta}$ Kronecker symbol, and state the following Theorem on the higher order derivative of power of trinomial.

\begin{thm}\label{MultDerPowerTR}
	$\forall m\in\mathbb{N}$, $\forall a,b,c,x\in\mathbb{R}$, we set
	
\begin{equation}\label{key}
 \partial _{x}^{m} (a\, x^{2} +b\, x+c)^{n} =
m!\, \sum _{r=0}^{\left\lfloor \frac{m}{2} \right\rfloor }\frac{(2\, ax+b)^{m-2r} a^{r} }{(m-2r)!r!} \frac{n!}{(n-m+r)!}(a\, x^{2} +b\, x+c)^{n-m+r}.
\end{equation}
\end{thm}
\begin{proof}
By using the umbral operator \eqref{ppi}, we can write

\begin{equation}\label{key}
x^m=e^{\hat{p}_m x}\;\pi_0=\sum_{r=0}^\infty \frac{x^r}{r!}r!\delta_{r,m}
\end{equation}
and then
\begin{equation*}\label{key}
\begin{split}
 \partial _{x}^{m} (a\, x^{2} +b\, x+c)^{n} &= \partial _{x}^{m} e^{\hat{p}_n (a\, x^{2} +b\, x+c)}\;\pi_0=H_m((2ax+b)\hat{p}_n,a\hat{p_n})e^{\hat{p}_n (a\, x^{2} +b\, x+c)}\;\pi_0=\\
 & =m!\, \sum _{r=0}^{\left\lfloor \frac{m}{2} \right\rfloor }\frac{(2\, ax+b)^{m-2r} a^{r} }{(m-2r)!r!} \sum_{s=0}^\infty \frac{(a\, x^{2} +b\, x+c)^{s} }{s!}\hat{p}_n^{s+m-r}\pi_0=\\
 & =m!\, \sum _{r=0}^{\left\lfloor \frac{m}{2} \right\rfloor }\frac{(2\, ax+b)^{m-2r} a^{r} }{(m-2r)!r!} \sum_{s=0}^\infty \frac{(a\, x^{2} +b\, x+c)^{s} }{s!}(s+m-r)!\delta_{s+m-r,n}=\\
 & =m!\, \sum _{r=0}^{\left\lfloor \frac{m}{2} \right\rfloor }\frac{(2\, ax+b)^{m-2r} a^{r} }{(m-2r)!r!}\frac{n!}{(n-m+r)!}(a\, x^{2} +b\, x+c)^{n-m+r}.
 \end{split}
\end{equation*}
\end{proof}
%The Theorem \eqref{MultDerPowerTR} opens the way to infinite new results...\\

This concluding example opens the way to further use of the methods, we have treated so far, in a wider context. \\

In this article we have dealt with a description of umbral, operational and integral techniques, we have tried to present a unitary point of view. The results of this paper have fixed the elements for our future researches, including Combinatorics, Group Theory and Invariant Theory.\\

\textbf{Acknowledgements}\\

\noindent The work of Dr. S. Licciardi was supported by an Enea Research Center individual fellowship.\\

\textbf{Author Contributions}\\

\noindent Conceptualization: G.D.; methodology: G.D., S.L.; data curation: S.L.; validation:  G.D., S.L.; formal analysis: G.D., S.L.; writing - original draft preparation: G.D., S.L.; writing - review and editing: S.L. .\\

{}

\end{document}